\documentclass{amsart}[12pt]

\usepackage{amsmath}
\usepackage{amsfonts}
\usepackage{amssymb}

\newtheorem{theorem}{Theorem}
\newtheorem{lemma}{Lemma}
\newtheorem{proposition}{Proposition}
\newtheorem{corollary}{Corollary}

\theoremstyle{definition}

\newtheorem{remark}{Remark}

\begin{document}
\title[Volterra type operators]%
{Volterra type operators on weighted Banach spaces of analytic functions}

\author{Qingze Lin}

\address{School of Applied Mathematics, Guangdong University of Technology, Guangzhou, Guangdong, 510520, P.~R.~China}\email{gdlqz@e.gzhu.edu.cn}

\begin{abstract}
Smith et al. recently gave the sufficient and necessary conditions for the boundedness of Volterra type operators on Banach spaces of bounded analytic functions when the symbol functions are univalent. In this paper, we give the complete characterizations of the conditions for the boundedness and compactness of Volterra type operators $T_g$ and their companion operators $S_g$ between weighted Banach spaces of analytic functions, which essentially generalize their works.
\end{abstract}

\maketitle

\section{\bf Introduction}
Let $\mathbb{D}$ be the unit disk of the complex plane $\mathbb{C}$ and $H(\mathbb D)$ be the space consisting of all analytic functions on the unit disk. For any analytic function $g\in H(\mathbb D)$, the Volterra type operator $T_g$ is defined by
$$(T_gf)(z)=\int_0^z f(\omega)g'(\omega)d\omega,\quad z\in\mathbb D, f\in H(\mathbb D)\,.$$
Its companion operator $S_g$ is defined by
$$(S_gf)(z)=\int_0^z f'(\omega)g(\omega)d\omega,\quad z\in\mathbb D, f\in H(\mathbb D)\,.$$

The boundedness and compactness of these two operators on some spaces of analytic functions were extensively studied. Pommerenke \cite{P} firstly studied the boundedness of $T_g$ on Hardy-Hilbert space $H^2$. After his work, Aleman, Siskakis and Cima \cite{AC,AS} studied the boundedness and compactness of $T_g$ on Hardy space $H^p$, in which they showed that $T_g$ is bounded (compact) on $H^p,\ 0<p<\infty$, if and only if $g\in BMOA\ (g\in VMOA)$. For the related works, see \cite{LMN}. Here, $BMOA\text{ and }VMOA$ denote the spaces of $f\in H(\mathbb{D})$ with boundary values of bounded mean oscillation and vanishing mean oscillation, respectively. What's more, Aleman and Siskakis \cite{AS1} studied the boundedness and compactness of $T_g$ on Bergman spaces and Xiao \cite{XJ} studied $T_g$ and $S_g$ on $Q_p$ spaces.

Recently, Li et al. \cite{LLZ} studied $T_g$ and $S_g$ on analytic Morrey spaces. Lin et al. \cite{LIN} characterized the boundedness of $T_g$ and $S_g$ acting on the derivative Hardy spaces $S^p(1\leq p<\infty)$. Miihkinen \cite{SM} investigated the strict singularity of $T_g$ on Hardy space $H^p$ and showed that the strict singularity of $T_g$ coincides with its compactness on $H^p$. Mengestie \cite{TM} obtained a complete description of the boundedness and compactness of the product of Volterra type operators and composition operators on weighted Fock spaces. Furthermore, by applying the Carleson embedding theorem and the Littlewood-Paley formula, Constantin and Pel\'{a}ez \cite{CP} obtained the boundedness and compactness of $T_g$ on weighted Fock spaces and investigated the invariant subspaces of the Volterra operator $T_z$ on such spaces.

Now, for any $\alpha\geq0$, we define the weighted \mbox{Banach} spaces $H^{\infty}_\alpha$ of analytic functions given by
$$H^{\infty}_\alpha=\{f\in H(\mathbb D):\ \|f\|_{H^{\infty}_\alpha}:=\sup_{z\in \mathbb D}\{(1-|z|^2)^\alpha|f(z)|\}<\infty\}\,.$$
When $\alpha=0$, $H^{\infty}_0$ is the usual \mbox{Banach} space of bounded analytic functions. The norms and essential norms of several bounded linear operators on such spaces had been extensively studied. For instance, Montes-Rodr\'{\i}guez \cite{AM} obtained the essential norms of weighted composition operators on $H^{\infty}_\alpha$.

Although the necessary and sufficient conditions for $T_g: H^{\infty}_0\rightarrow H^{\infty}_0$ to be bounded are characterized recently by \mbox{Contreras}, \mbox{Pel\'{a}ez}, \mbox{Pommerenke} and R\"{a}tty\"{a} \cite{CPPR}, the conditions characterizing the symbol $g$ use the Cauchy transforms which are very difficult to verified. However, Anderson, Jovovic and Smith \cite{AJS} conjectured that $T_g: H^{\infty}_0\rightarrow H^{\infty}_0$ is bounded if and only if
$$\sup_{0\leq\theta<2\pi}\int_0^1|g'(re^{i\theta})|dr<+\infty\,.$$
While this condition of symbol $g$ is sufficient for the boundedness of $T_g: H^{\infty}_0\rightarrow H^{\infty}_0$, it was recently proved to be not necessary by Smith, Stolyarov and Volberg in \cite{SSV}, where a counterexample was given. However, when the symbol $g$ is univalent, this conjecture was proved to be affirmative in the same paper. Following their works, Eklund et. al. \cite{ELPSW} studied the boundedness and compactness of $T_g$ between $H_{\alpha}^\infty$ and $H^{\infty}_0$, where $0\leq\alpha<1$ and $g$ is univalent. It should be noticed that Basallote et al. \cite{BCHMP} characterized the boundedness, compactness and weak compactness of $T_g$ acting between different weighted Banach spaces of analytic functions with the quasi-normal weight, and what's more, they applied the characterization of compactness to analyze the behavior of semigroups of composition operators on such spaces.

In this paper, using the ideas from \cite{ELPSW,SSV}, we give the sufficient and necessary conditions for the boundedness and compactness of $T_g: H^{\infty}_\alpha\rightarrow H^{\infty}_\beta$ and $S_g: H^{\infty}_\alpha\rightarrow H^{\infty}_\beta$ when $0\leq\alpha,\beta$, which generalize their works.

\section{\bf The boundedness of Volterra type operators}
For any $0<\eta<\pi, 0<r<1$, and $0\leq \theta<2\pi$, denote by $\mathcal{B}(\Omega^{r}_{\eta,\theta})$ the set of all analytic functions $F$ defined in the open sector
$$\Omega^{r}_{\eta,\theta}:=\left\{z\in\mathbb{D}: 0<|z|<r,\ \theta-\frac{\eta}{2}<\arg{z}<\theta+\frac{\eta}{2}\right\}\,,$$
such that
$$|F'(z)|\leq\frac{C_F}{|z|},\quad\text{for } z\in \Omega^{r}_{\eta,\theta}\,.$$
Here, $C_F$ is a constant depending only on function $f$.

By the similar arguments from \cite{ELPSW,SSV}, we obtain the following two lemmas.
\begin{lemma}\label{le1}
Let $0<\gamma<\eta<\pi$, $0\leq \theta<2\pi$ and $\varepsilon>0$. Then there exists a positive number $\delta(\varepsilon)$ such that for any $F\in \mathcal{B}(\Omega^{1/2}_{\gamma,\theta})$, there is a real harmonic function $u$ defined in $\Omega^{1}_{\eta,\theta}$ such that:

(1) $\left|\text{Re}(F(xe^{i\theta}))-u(xe^{i\theta})\right|\leq\varepsilon$,\quad for $x\in (0,\delta(\varepsilon)]$;

(2) $|\widetilde{u}(z)|\leq C(\varepsilon,\gamma,\eta,C_F)<\infty$,\quad for $z\in\Omega^{1}_{\eta,\theta}$, where $\widetilde{u}$ is the conjugate of $u$ such that $\widetilde{u}\left(\frac{1}{2}e^{i\theta}\right)=0$.
\end{lemma}

\begin{lemma}\label{le2}
Let $0<\gamma<\eta<\pi$, $0\leq \theta<2\pi$. If $\psi_{\eta,\theta}:\Omega^{1}_{\eta,\theta}\rightarrow\mathbb{D}$ is a conformal map such that $\psi_{\eta,\theta}\left(\frac{1}{2}e^{i\theta}\right)=0$ and $\psi_{\eta,\theta}(0)=e^{i\theta}$,
then there exists a constant $C_1(\gamma,\eta)$ such that
$$\frac{|\psi'_{\eta,\theta}(z)|}{1-|\psi_{\eta,\theta}(z)|^2}\leq\frac{C_1(\gamma,\eta)}{|z|}$$
for all $z\in \Omega^{1/2}_{\gamma,\theta}$\,.
\end{lemma}

Recall that the Bloch space $\mathcal{B}$ is defined by
$$\mathcal{B}=\{f\in H(\mathbb D):\ \|f\|_{\mathcal{B}}:=|f(0)|+\sup_{z\in \mathbb D}\{(1-|z|^2)|f'(z)|\}<\infty\}\,.$$

Now we are able to give the sufficient and necessary condition for the boundedness of Volterra type operators $T_g: H^{\infty}_\alpha\rightarrow H^{\infty}_\beta$\,.
\begin{theorem}\label{th1}
If $g\in H(\mathbb{D})$ such that $\log(g')\in \mathcal{B}$ and $0\leq\alpha,\beta$, then $T_g: H^{\infty}_\alpha\rightarrow H^{\infty}_\beta$ is bounded if and only if
$$\limsup_{t\rightarrow1^-}\sup_{0\leq\theta<2\pi}(1-t^2)^\beta\int_0^t\frac{|g'(re^{i\theta})|}{(1-r^2)^\alpha}dr<+\infty\,.$$
\end{theorem}
\begin{proof}
Assume first that $$\limsup_{t\rightarrow1^-}\sup_{0\leq\theta<2\pi}(1-t^2)^\beta\int_0^t\frac{|g'(re^{i\theta})|}{(1-r^2)^\alpha}dr=+\infty\,.$$
Then choose a nondecreasing positive sequence $\{t_{0,n}\}_{n=1}^\infty$, with $t_{0,n}\rightarrow1^-$, such that for each $t_{0,n}$, there is an angle $\theta_n$, with $0\leq\theta_n<2\pi$, such that
$$(1-t_{0,n}^2)^\beta\int_0^{t_{0,n}}\frac{|g'(re^{i\theta_n})|}{(1-r^2)^\alpha}dr\geq n\,.$$

Define the function $F_n:\Omega^{1}_{\eta,\theta_n}\rightarrow\mathbb{D}$ by
$$F_n(z)=-i\log\left(g'\circ\psi_{\eta,\theta_n}(z)\right)\,.$$
Since $log(g')\in\mathcal{B}$, there is a constant $K>0$ such that
$$|g''(z)|\leq \frac{K|g'(z)|}{(1-|z|^2)}\quad \text{for all } z\in \mathbb{D}\,.$$
Then by Lemma~\ref{le2}, we have
\begin{equation}\begin{split}\nonumber
|F'_n(z)|&=\frac{\left|g''(\psi_{\eta,\theta_n}(z))\right|}{\left|g'(\psi_{\eta,\theta_n}(z))\right|}\cdot\left|\psi'_{\eta,\theta_n}(z)\right|\leq\frac{K\left|\psi'_{\eta,\theta_n}(z)\right|}{1-\left|\psi_{\eta,\theta_n}(z)\right|^2}\leq\frac{K\cdot C_1(\gamma,\eta)}{|z|}\,,
\end{split}\end{equation}
for all $z\in \Omega^{1/2}_{\gamma,\theta_n}$\,.
Hence, the restriction of $F_n$ to $\Omega^{1/2}_{\gamma,\theta_n}$ belongs to $\mathcal{B}(\Omega^{1/2}_{\gamma,\theta_n})$\,.

By Lemma~\ref{le1} with $\varepsilon=\frac{\pi}{3}$, we have a positive number $\delta(\frac{\pi}{3})$ such that for any $F_n\in \mathcal{B}(\Omega^{1/2}_{\gamma,\theta_n})$, there is a real harmonic function $u_n$ defined in $\Omega^{1}_{\eta,\theta_n}$ such that:

(1) $\left|\text{Re}(F_n(xe^{i\theta_n}))-u_n(xe^{i\theta_n})\right|\leq\frac{\pi}{3}$,\quad for $x\in (0,\delta(\frac{\pi}{3})]$;

(2) $|\widetilde{u}_n(z)|\leq C(\frac{\pi}{3},\gamma,\eta,K\cdot C_1(\gamma,\eta))=:C_2(\gamma,\eta)<\infty$,\quad for $z\in\Omega^{1}_{\eta,\theta_n}$, where $\widetilde{u}_n$ is the conjugate of $u_n$ such that $\widetilde{u}_n\left(\frac{1}{2}e^{i\theta_n}\right)=0$.

From (1) and the definition of conformal function $\psi_{\eta,\theta_n}$, we see that there exists a $r_\eta$ with $0<r_\eta<1$ such that
$$\left|\arg{g'(re^{i\theta_n})}-u_n(\psi^{-1}_{\eta,\theta_n}(re^{i\theta_n}))\right|\leq\frac{\pi}{3},\quad \text{for }r\in [r_\eta,1)\,.$$

Now define
$$h_n(z):=e^{-i(u_n(z)+i\widetilde{u}_n(z))}\,,$$
and consider the sequence $\{f_n\}_{n=1}^\infty$ of functions defined by
$$f_n(z):=\frac{h_n\circ\psi^{-1}_{\eta,\theta_n}(z)}{(1-e^{-2\theta_ni}z^2)^\alpha}\,.$$
Then by (2), the sequence $\{f_n\}_{n=1}^\infty$ is bounded in $H^{\infty}_{\alpha}$ with
\begin{equation}\begin{split}\nonumber
\|f_n\|_{H^{\infty}_{\alpha}}&=\sup_{z\in\mathbb{D}}\left\{(1-|z|^2)^\alpha \left|\frac{h_n\circ\psi^{-1}_{\eta,\theta_n}(z)}{(1-e^{-2\theta_ni}z^2)^\alpha}\right|\right\}\\
&\leq\left\|h_n\circ\psi^{-1}_{\eta,\theta_n}(z)\right\|_{H^{\infty}_0}\\
&\leq e^{C_2(\gamma,\eta)}.
\end{split}\end{equation}

Now choose $n$ large enough such that $r_\eta\leq t_{0,n}<1$, we observe that
\begin{equation}\begin{split}\nonumber
\|(T_gf_n)\|_{H^{\infty}_{\beta}}&=\sup_{z\in\mathbb{D}}\left\{(1-|z|^2)^\beta \left|\int_0^zf_n(\zeta)g'(\zeta)d\zeta\right|\right\}\\
&\geq(1-t^2_{0,n})^\beta\left|\int_0^{t_{0,n}}f_n(re^{i\theta_n})g'(re^{i\theta_n})e^{i\theta_n}dr\right|\\
&\geq\left|\text{Re}\left((1-t^2_{0,n})^\beta\int_{r_\eta}^{t_{0,n}}f_n(re^{i\theta_n})g'(re^{i\theta_n})dr\right)\right|\\
&\phantom{\geq}-\left|\text{Re}\left((1-t^2_{0,n})^\beta\int_0^{r_\eta}f_n(re^{i\theta_n})g'(re^{i\theta_n})dr\right)\right|
\end{split}\end{equation}

First, we see that
\begin{equation}\begin{split}\nonumber
\left|\text{Re}\left((1-t^2_{0,n})^\beta\int_0^{r_\eta}f_n(re^{i\theta_n})g'(re^{i\theta_n})dr\right)\right|&\leq\int_0^{r_\eta}|f_n(re^{i\theta_n})||g'(re^{i\theta_n})|dr\\
&\leq\|f_n\|_{H^{\infty}_{\alpha}}\int_0^{r_\eta}\frac{|g'(re^{i\theta_n})|}{(1-r^2)^\alpha}dr\\
&\leq\frac{e^{C_2(\gamma,\eta)}C_3(\eta)}{(1-r^2_\eta)^\alpha}\,,
\end{split}\end{equation}
where $C_3(\eta):=\sup_{|z|\leq r_\eta}\{|g'(z)|\}$\,.

Next, to estimate another integral, we see that
\begin{equation}\begin{split}\nonumber
&\phantom{=\;}\left|\text{Re}\left((1-t^2_{0,n})^\beta\int_{r_\eta}^{t_{0,n}}f_n(re^{i\theta_n})g'(re^{i\theta_n})dr\right)\right|\\
&=\left|(1-t^2_{0,n})^\beta\text{Re}\left(\int_{r_\eta}^{t_{0,n}}\frac{g'(re^{i\theta_n})}{(1-r^2)^\alpha}e^{-i\left(u_n(\psi^{-1}_{\eta,\theta_n}(re^{i\theta_n}))+i\widetilde{u}_n(\psi^{-1}_{\eta,\theta_n}(re^{i\theta_n}))\right)}dr\right)\right|\\
&=\Bigg|(1-t^2_{0,n})^\beta\int_{r_\eta}^{t_{0,n}}\frac{|g'(re^{i\theta_n})|}{(1-r^2)^\alpha}e^{\widetilde{u}_n\left(\psi^{-1}_{\eta,\theta_n}(re^{i\theta_n})\right)}\times\\
&\phantom{=\;}\text{Re}\left(e^{i\left(\arg(g'(re^{i\theta_n}))-u_n(\psi^{-1}_{\eta,\theta_n}(re^{i\theta_n}))\right)}\right)dr\Bigg|\\
&\geq\cos(\frac{\pi}{3})e^{-C_2(\gamma,\eta)}\left|(1-t^2_{0,n})^\beta\int_{r_\eta}^{t_{0,n}}\frac{|g'(re^{i\theta_n})|}{(1-r^2)^\alpha}dr\right|\\
&\geq\cos(\frac{\pi}{3})e^{-C_2(\gamma,\eta)}(1-t^2_{0,n})^\beta\left(\left|\int_{0}^{t_{0,n}}\frac{|g'(re^{i\theta_n})|}{(1-r^2)^\alpha}dr\right|-\left|\int_{0}^{r_\eta}\frac{|g'(re^{i\theta_n})|}{(1-r^2)^\alpha}dr\right|\right)\\
&\geq\cos(\frac{\pi}{3})e^{-C_2(\gamma,\eta)}\left(\left|(1-t^2_{0,n})^\beta\int_{0}^{t_{0,n}}\frac{|g'(re^{i\theta_n})|}{(1-r^2)^\alpha}dr\right|-\left|\int_{0}^{r_\eta}\frac{|g'(re^{i\theta_n})|}{(1-r^2)^\alpha}dr\right|\right)\\
&\geq\cos(\frac{\pi}{3})e^{-C_2(\gamma,\eta)}\left(n-\left|\int_{0}^{r_\eta}\frac{|g'(re^{i\theta_n})|}{(1-r^2)^\alpha}dr\right|\right)\\
&\geq\cos(\frac{\pi}{3})e^{-C_2(\gamma,\eta)}\left(n-\frac{C_3(\eta)}{(1-r^2_\eta)^\alpha}\right)\,.
\end{split}\end{equation}
Hence, for $n$ large enough such that $r_\eta\leq t_{0,n}<1$, we have
$$\|(T_gf_n)\|_{H^{\infty}_{\beta}}\geq\cos(\frac{\pi}{3})e^{-C_2(\gamma,\eta)}n-e^{-C_2(\gamma,\eta)}\left(e^{2C_2(\gamma,\eta)}+\cos(\frac{\pi}{3})\right)\frac{C_3(\eta)}{(1-r^2_\eta)^\alpha}\,.$$
Letting $n\rightarrow\infty$, it follows that $T_g: H^{\infty}_\alpha\rightarrow H^{\infty}_\beta$ is not bounded.

Conversely, suppose that
$$\limsup_{t\rightarrow1^-}\sup_{0\leq\theta<2\pi}(1-t^2)^\beta\int_0^t\frac{|g'(re^{i\theta})|}{(1-r^2)^\alpha}dr<+\infty\,,$$
then there exists $N>0$, $0<t_0<1$ such that
$$\sup_{0\leq\theta<2\pi}(1-t^2)^\beta\int_0^t\frac{|g'(re^{i\theta})|}{(1-r^2)^\alpha}dr\leq N\quad \text{whenever }t_0<t<1\,.$$
For any $f\in H^{\infty}_\alpha$, we have
\begin{equation}\begin{split}\nonumber
\|T_g(f)\|_{H^{\infty}_{\beta}}&=\sup_{z\in\mathbb{D}}\left\{(1-|z|^2)^\beta \left|\int_0^zf(\zeta)g'(\zeta)d\zeta\right|\right\}\\
&=\sup_{0\leq\theta<2\pi}\sup_{0\leq R<1}\left\{(1-R^2)^\beta\left|\int_0^{Re^{i\theta}}f(\zeta)g'(\zeta)d\zeta\right|\right\}\\
&=\sup_{0\leq\theta<2\pi}\sup_{0\leq R<1}\left\{(1-R^2)^\beta\left|\int_0^{R}f(re^{i\theta})g'(re^{i\theta})e^{i\theta}dr\right|\right\}\\
&\leq\sup_{0\leq\theta<2\pi}\sup_{0\leq R\leq t_0}\left\{(1-R^2)^\beta\int_0^{R}|f(re^{i\theta})||g'(re^{i\theta})|dr\right\}\\
&\phantom{\leq}+\sup_{0\leq\theta<2\pi}\sup_{t_0< R<1}\left\{(1-R^2)^\beta\int_0^{R}|f(re^{i\theta})||g'(re^{i\theta})|dr\right\}\\
&\leq \|f\|_{H^{\infty}_{\alpha}}\sup_{0\leq\theta<2\pi}\sup_{0\leq R\leq t_0}\left\{(1-R^2)^\beta\int_0^{R}\frac{|g'(re^{i\theta})|}{(1-r^2)^\alpha}dr\right\}\\
&\phantom{\leq}+\|f\|_{H^{\infty}_{\alpha}}\sup_{0\leq\theta<2\pi}\sup_{t_0< R<1}\left\{(1-R^2)^\beta\int_0^{R}\frac{|g'(re^{i\theta})|}{(1-r^2)^\alpha}dr\right\}\\
&\leq \|f\|_{H^{\infty}_{\alpha}}(M_{t_0}+N)\,,
\end{split}\end{equation}
where $M_{t_0}:=\sup_{0\leq\theta<2\pi}\sup_{0\leq R\leq t_0}\left\{(1-R^2)^\beta\int_0^{R}\frac{|g'(re^{i\theta})|}{(1-r^2)^\alpha}dr\right\}$\,.
Accordingly, $T_g: H^{\infty}_\alpha\rightarrow H^{\infty}_\beta$ is bounded.
\end{proof}

\begin{remark}\label{re1}
From the proof of Theorem~\ref{th1}, we see that the condition $$\limsup_{t\rightarrow1^-}\sup_{0\leq\theta<2\pi}(1-t^2)^\beta\int_0^t\frac{|g'(re^{i\theta})|}{(1-r^2)^\alpha}dr<+\infty$$
is sufficient for the boundedness of $T_g: H^{\infty}_\alpha\rightarrow H^{\infty}_\beta$ without the requirement that $\log(g')\in \mathcal{B}$.
\end{remark}

\begin{remark}\label{re2}
If $g\in H(\mathbb{D})$ is univalent, then by \cite{PB}, it holds that $\log(g')\in \mathcal{B}$, thus Theorem~\ref{th1} holds for the univalent case.
\end{remark}

\begin{remark}\label{re3}
It should be noticed that Basallote et. al. \cite{BCHMP} had given the following complete characterization of the boundedness of $T_g: H^{\infty}_\alpha\rightarrow H^{\infty}_\beta$ for the general symbol $g\in H(\mathbb{D})$ when $\beta>0$: $T_g: H^{\infty}_\alpha\rightarrow H^{\infty}_\beta$ is bounded if and only if
$$\sup_{z\in\mathbb{D}}(1-|z|^2)^{\beta+1-\alpha}|g'(z)|<\infty\,.$$
\end{remark}

As a corollary, we give the following result obtained by Smith, Stolyarov and Volberg in \cite{SSV}.
\begin{corollary}\label{cor1}
Let $g\in H(\mathbb{D})$ such that $\log(g')\in \mathcal{B}$. Then $T_g: H^{\infty}_0\rightarrow H^{\infty}_0$ is bounded if and only if
$$\sup_{0\leq\theta<2\pi}\int_0^1|g'(re^{i\theta})|dr<+\infty\,.$$
\end{corollary}
\begin{proof}
By Theorem~\ref{th1},  $T_g: H^{\infty}_0\rightarrow H^{\infty}_0$ is bounded if and only if
$$\limsup_{t\rightarrow1^{-}}\sup_{0\leq\theta<2\pi}\int_0^t|g'(re^{i\theta})|dr<+\infty\,.$$
Since $\int_0^t|g'(re^{i\theta})|dr$ is increasing as $t\rightarrow1^{-}$, it follows that
\begin{equation}\begin{split}\nonumber
\limsup_{t\rightarrow1^{-}}\sup_{0\leq\theta<2\pi}\int_0^t|g'(re^{i\theta})|dr&=\sup_{0\leq\theta<2\pi}\limsup_{t\rightarrow1^{-}}\int_0^t|g'(re^{i\theta})|dr\\
&=\sup_{0\leq\theta<2\pi}\int_0^1|g'(re^{i\theta})|dr.
\end{split}\end{equation}
\end{proof}

Recall that the disk algebra $\mathcal{A}$ is the space consisting of all analytic functions $f$ defined on $\mathbb D$ which can be continuously extended to $\overline{\mathbb D}$ and endowed with the norm $\|f\|=\sup_{z\in\mathbb D}|f(z)|$\,.

\begin{corollary}\label{cor2}
Let $g\in H(\mathbb{D})$ such that $\log(g')\in \mathcal{B}$. Then $T_g$ acting on the disk algebra $\mathcal{A}$ is bounded if and only if $g\in\mathcal{A}$ and
$$\sup_{0\leq\theta<2\pi}\int_0^1|g'(re^{i\theta})|dr<+\infty\,.$$
\end{corollary}
\begin{proof}
From \cite[Proposition~2.17]{AJS} we see that for $g\in H(\mathbb{D})$, $T_g: \mathcal{A}\rightarrow \mathcal{A}$ is bounded if and only if $g\in \mathcal{A}$ and $T_g: H^{\infty}_0\rightarrow H^{\infty}_0$ is bounded. Thus, the corollary follows from Corollary~\ref{cor1}.
\end{proof}

\begin{proposition}\label{pro1}
If $g\in H(\mathbb{D})$, then the following statements are equivalent:

(1) $S_g: \mathcal{A}\rightarrow \mathcal{A}$ is bounded;

(2) $S_g: H^{\infty}_0\rightarrow H^{\infty}_0$ is bounded;

(3) $T_g: H^{\infty}_0\rightarrow H^{\infty}_0$ is bounded.
\end{proposition}
\begin{proof}
From \cite{AJS} we see that for $g\in H(\mathbb{D})$, $S_g: H^{\infty}_0\rightarrow H^{\infty}_0$ is bounded if and only if $T_g: H^{\infty}_0\rightarrow H^{\infty}_0$ is bounded.

Assume first that $S_g: H^{\infty}_0\rightarrow H^{\infty}_0$ is bounded, then for any polynomial $p$, since $gp'\in H^{\infty}_0$, we have
$$(S_gp)(z)=\int_0^zg(\zeta)p'(\zeta)d\zeta\in \mathcal{A}\,.$$
It is well known that polynomials are dense in $\mathcal{A}$, it follows that $S_g: \mathcal{A}\rightarrow \mathcal{A}$ is bounded.

Conversely, assume that $S_g: \mathcal{A}\rightarrow \mathcal{A}$ is bounded. Let $f\in H^{\infty}_0$ and $f_r(z)=f(rz)$ for $0<r<1$, then we have $f_r\in \mathcal{A}$\,. For any $z\in \mathbb{D}$,
$$|(S_gf)(z)|=\left|\lim_{r\rightarrow1^-}(S_gf_r)(z)\right|\leq\lim_{r\rightarrow1^-}\|S_g\|\|f_r\|_{H^{\infty}_0}\leq\|S_g\|\|f\|_{H^{\infty}_0}\,,$$
which implies that $S_g: H^{\infty}_0\rightarrow H^{\infty}_0$ is bounded.
\end{proof}

\begin{corollary}\label{cor3}
Let $g\in H(\mathbb{D})$ such that $\log(g')\in \mathcal{B}$. Then $S_g$ acting on the disk algebra $\mathcal{A}$ (or $H^{\infty}_0$) is bounded if and only if
$$\sup_{0\leq\theta<2\pi}\int_0^1|g'(re^{i\theta})|dr<+\infty\,.$$
\end{corollary}

Now for the boundedness of the companion operators $S_g: H^{\infty}_\alpha\rightarrow H^{\infty}_\beta$, due to Corollary~\ref{cor3}, we just need to consider the case when $\alpha>0$. In this case, since the weight $(1-|z|^2)^\alpha$ is normal, we have $H^{\infty}_\alpha=\mathcal{B}^{\infty}_{\alpha+1}$, where $\mathcal{B}^{\infty}_{\alpha}$ is a weighted \mbox{Bloch} space defined as follows:
$$\mathcal{B}^{\infty}_{\alpha}=\{f\in H(\mathbb{D}):\ \|f\|_{\mathcal{B}^{\infty}_{\alpha}}=\sup_{z\in\mathbb{D}}(1-|z|^2)^\alpha|f'(z)|<\infty\}\,.$$
Keep in mind that $H^{\infty}_\alpha=\mathcal{B}^{\infty}_{\alpha+1}$ and consider the sequence $\{g_n\}_{n=1}^\infty$ of functions defined by
$$g_n(z):=\int_0^z\frac{f_n(\zeta)d\zeta}{(1-e^{-2\theta_ni}\zeta^2)}\,,$$
with $f_n$ defined in the proof of Theorem~\ref{th1}, we have
\begin{theorem}\label{th2}
Let $g\in H(\mathbb{D})$ such that $\log(g)\in \mathcal{B}$ and $0<\alpha,0\leq\beta$. Then $S_g: H^{\infty}_\alpha\rightarrow H^{\infty}_\beta$ is bounded if and only if
$$\limsup_{t\rightarrow1^-}\sup_{0\leq\theta<2\pi}(1-t^2)^\beta\int_0^t\frac{|g(re^{i\theta})|}{(1-r^2)^{\alpha+1}}dr<+\infty\,.$$
\end{theorem}

The following corollary is the most interesting part of Theorem~\ref{th2}.
\begin{corollary}\label{cor4}
Let $g\in H(\mathbb{D})$ such that $\log(g)\in \mathcal{B}$ and $0<\alpha$\,. Then $S_g: H^{\infty}_\alpha\rightarrow H^{\infty}_0$ is bounded if and only if
$$\sup_{0\leq\theta<2\pi}\int_0^1\frac{|g(re^{i\theta})|}{(1-r^2)^{\alpha+1}}dr<+\infty\,.$$
\end{corollary}

If $\beta>0$, then we can give the following complete characterization of the boundedness of $S_g: H^{\infty}_\alpha\rightarrow H^{\infty}_\beta$ for the general symbol $g\in H(\mathbb{D})$\,.
\begin{theorem}\label{th3}
If $g\in H(\mathbb{D})$ and $0\leq\alpha,0<\beta$, then $S_g: H^{\infty}_\alpha\rightarrow H^{\infty}_\beta$ is bounded if and only if
$$\sup_{z\in\mathbb{D}}(1-|z|^2)^{\beta-\alpha}|g(z)|<\infty\,.$$
\end{theorem}
\begin{proof}
If $0<\beta$, then $H^{\infty}_\beta=\mathcal{B}^{\infty}_{\beta+1}$\,. Thus the boundedness of $S_g: H^{\infty}_\alpha\rightarrow H^{\infty}_\beta$ is equivalent to the boundedness of $M_gD: H^{\infty}_\alpha\rightarrow H^{\infty}_{\beta+1}$, where $M_g$ is the multiplication operator and $D$ is the differentiation operator, respectively. Now by choosing $\varphi(z)\equiv z$ in \cite[Theorem~7]{MZ}, the boundedness of $M_gD: H^{\infty}_\alpha\rightarrow H^{\infty}_{\beta+1}$ is equivalent to the following condition
$$\sup_{z\in\mathbb{D}}(1-|z|^2)^{\beta-\alpha}|g(z)|<\infty\,.$$
\end{proof}

\section{\bf The compactness of Volterra type operators}
In this section, we firstly characterize the compactness of $T_g: H^{\infty}_\alpha\rightarrow H^{\infty}_\beta$ and then consider the corresponding questions of its companion operator $S_g: H^{\infty}_\alpha\rightarrow H^{\infty}_\beta$\,. For $0\leq\alpha$, we introduce the notation:
$$\overline{B}:=\ \{f\in H^{\infty}_\alpha:\ \|f\|_{H^{\infty}_\alpha}\leq1\}$$

By standard arguments, we prove the following lemma.
\begin{lemma}\label{le3}
Let $0\leq\alpha,\beta$. If  $T_g: H^{\infty}_\alpha\rightarrow H^{\infty}_\beta$ is bounded, then the following two statements are equivalent:

(1) $T_g: H^{\infty}_\alpha\rightarrow H^{\infty}_\beta$ is compact;

(2) For  any bounded sequence $\{f_n\}_{n=1}^\infty\subset H^{\infty}_\alpha$ such that $f_n(z)\rightarrow0$ uniformly on any compact subset of $\mathbb{D}$ as $n\rightarrow\infty$, it holds that $\|T_g(f_n)\|_{H^{\infty}_\beta}\rightarrow0$ as $n\rightarrow\infty$.
\end{lemma}
\begin{proof}
Assume that (2) holds. Let $\{k_n\}_{n=1}^\infty\subset\overline{B}$, then $\{k_n\}$ is a normal family and thus, there exists $k\in H^{\infty}_\alpha$ and a subsequence $\{k_{n_j}\}_{j=1}^\infty$ such that $k_{n_j}\rightarrow k$ uniformly on any compact subset of $\mathbb{D}$ as $j\rightarrow\infty$. Let $\{f_j\}_{j=1}^\infty=\{k_{n_j}-k\}_{j=1}^\infty$. By (2), $\|T_g(f_j)\|_{H^{\infty}_\beta}=\|T_g(k_{n_j})-T_g(k)\|_{H^{\infty}_\beta}\rightarrow0$ as $j\rightarrow\infty$. Since by assumption $T_g: H^{\infty}_\alpha\rightarrow H^{\infty}_\beta$ is bounded, it follows that $T_g(k)\in H^{\infty}_\beta$. Accordingly, $T_g(k_{n_j})\rightarrow T_g(k)$ in $H^{\infty}_\beta$ as $j\rightarrow\infty$, which implies that $T_g: H^{\infty}_\alpha\rightarrow H^{\infty}_\beta$ is compact.

Conversely, assume that $\{f_n\}_{n=1}^\infty\subset\overline{B}$, that $f_{n}\rightarrow 0$ uniformly on any compact subset of $\mathbb{D}$ as $n\rightarrow\infty$, and that $T_g: H^{\infty}_\alpha\rightarrow H^{\infty}_\beta$ is compact. To prove (2), it suffices to show that there exists a subsequence $\{f_{n_j}\}_{j=1}^\infty$ such that $f_{n_j}\rightarrow 0$ as $j\rightarrow\infty$. It is easy to show that $T_g(f_n)\rightarrow0$ uniformly on any compact subset of $\mathbb{D}$ as $n\rightarrow\infty$. Since $T_g: H^{\infty}_\alpha\rightarrow H^{\infty}_\beta$ is compact, there is a subsequence $\{f_{n_j}\}_{j=1}^\infty$ and $f\in H^{\infty}_\beta$ such that $\|T_g(f_{n_j})-f\|_{H^{\infty}_\beta}\rightarrow 0$ as $j\rightarrow\infty$. Since $T_g(f_{n_j})\rightarrow0$ uniformly on any compact subset of $\mathbb{D}$ as $j\rightarrow\infty$, it follows that $f=0$ and hence $\|T_g(f_{n_j})\|_{H^{\infty}_\beta}\rightarrow 0$ as $j\rightarrow\infty$. This completes the proof.
\end{proof}

\begin{theorem}\label{th4}
Let $g\in H(\mathbb{D})$ such that $\log(g')\in \mathcal{B}$ and $0\leq\alpha,\beta$. Then $T_g: H^{\infty}_\alpha\rightarrow H^{\infty}_\beta$ is compact if and only if
$$\lim_{t_2\rightarrow1^-}\limsup_{t_1\rightarrow1^-}\sup_{0\leq\theta<2\pi}(1-t^2_1)^\beta\int_{t_2}^{t_1}\frac{|g'(re^{i\theta})|}{(1-r^2)^\alpha}dr=0\,.$$
\end{theorem}
\begin{proof}
Assume first that
$$\lim_{t_2\rightarrow1^-}\limsup_{t_1\rightarrow1^-}\sup_{0\leq\theta<2\pi}(1-t^2_1)^\beta\int_{t_2}^{t_1}\frac{|g'(re^{i\theta})|}{(1-r^2)^\alpha}dr=0\,.$$
To prove that $T_g: H^{\infty}_\alpha\rightarrow H^{\infty}_\beta$ is compact, choose any sequence $\{f_n\}_{n=1}^\infty\subset H^{\infty}_{\alpha}$ such that there exists a positive number $W$ such that $\sup_{n\in \mathbb{N}}\|f_n\|_{H^{\infty}_{\alpha}}\leq W$ and $f_n\rightarrow 0$ uniformly on any compact subset of $\mathbb{D}$ as $n\rightarrow\infty$\,. Let $\epsilon>0$, then there exists $t_{2,\epsilon}$ with $0<t_{2,\epsilon}<1$ such that
$$\limsup_{t_1\rightarrow1^-}\sup_{0\leq\theta<2\pi}(1-t^2_1)^\beta\int_{t_{2,\epsilon}}^{t_1}\frac{|g'(re^{i\theta})|}{(1-r^2)^\alpha}dr<\epsilon\,.$$
Moreover, there exists $t_{1,\epsilon}$ with $t_{2,\epsilon}<t_{1,\epsilon}<1$ such that
$$\sup_{0\leq\theta<2\pi}(1-t^2_1)^\beta\int_{t_{2,\epsilon}}^{t_1}\frac{|g'(re^{i\theta})|}{(1-r^2)^\alpha}dr<2\epsilon,\quad \text{whenever } t_{1,\epsilon}<t_1<1\,.$$

Now choose $N_\epsilon>0$ such that
$$\sup_{|z|\leq t_{1,\epsilon}}|f_n(z)|<\frac{\epsilon}{M_{1,\epsilon}\cdot t_{1,\epsilon}},\quad \text{whenever } n>N_\epsilon\,,$$
where $M_{1,\epsilon}:=\sup_{0\leq|z|\leq t_{1,\epsilon}}\left\{|g'(z)|\right\}$\,.

Then for $n>N_\epsilon$, we have
\begin{equation}\begin{split}\nonumber
\|T_g(f_n)\|_{H^{\infty}_{\beta}}&=\sup_{0\leq\theta<2\pi}\sup_{0\leq R<1}\left\{(1-R^2)^\beta\left|\int_0^{R}f_n(re^{i\theta})g'(re^{i\theta})e^{i\theta}dr\right|\right\}\\
&\leq\sup_{0\leq\theta<2\pi}\sup_{0\leq R<1}\left\{(1-R^2)^\beta\int_0^{R}|f_n(re^{i\theta})||g'(re^{i\theta})|dr\right\}\\
&\leq\sup_{0\leq\theta<2\pi}\sup_{0\leq R\leq t_{1,\epsilon}}\left\{(1-R^2)^\beta\int_0^{R}|f_n(re^{i\theta})||g'(re^{i\theta})|dr\right\}\\
&\phantom{\leq}+\sup_{0\leq\theta<2\pi}\sup_{t_{1,\epsilon}<R}\left\{(1-R^2)^\beta\int_0^{t_{2,\epsilon}}|f_n(re^{i\theta})||g'(re^{i\theta})|dr\right\}\\
&\phantom{\leq}+\sup_{0\leq\theta<2\pi}\sup_{t_{1,\epsilon}<R}\left\{(1-R^2)^\beta\int_{t_{2,\epsilon}}^R|f_n(re^{i\theta})||g'(re^{i\theta})|dr\right\}\\
&\leq \epsilon+\epsilon+2W\epsilon=2(1+W)\epsilon\,,
\end{split}\end{equation}
which implies that $\lim_{n\rightarrow\infty}\|T_g(f_n)\|_{H^{\infty}_{\beta}}=0$ and hence, by Lemma~\ref{le3}, $T_g: H^{\infty}_\alpha\rightarrow H^{\infty}_\beta$ is compact.

Conversely, suppose that $T_g: H^{\infty}_\alpha\rightarrow H^{\infty}_\beta$ is compact and so it is bounded.
If
$$c:=\lim_{t_2\rightarrow1^-}\limsup_{t_1\rightarrow1^-}\sup_{0\leq\theta<2\pi}(1-t^2_1)^\beta\int_{t_2}^{t_1}\frac{|g'(re^{i\theta})|}{(1-r^2)^\alpha}dr>0\,,$$
then by Theorem~\ref{th1}, we see that $c<+\infty$\,.

Let $\{t_{2,n}\}_{n=1}^\infty$ be a sequence such that $0<t_{2,n}<t_{2,n+1}<1$, $\lim_{n\rightarrow\infty}t_{2,n}=1$ and
\begin{equation}\begin{split}\nonumber
c&=\inf_{n\in\mathbb{N}}\limsup_{t_1\rightarrow1^-}\sup_{0\leq\theta<2\pi}(1-t^2_1)^\beta\int_{t_{2,n}}^{t_1}\frac{|g'(re^{i\theta})|}{(1-r^2)^\alpha}dr\\
&=\lim_{n\rightarrow\infty}\limsup_{t_1\rightarrow1^-}\sup_{0\leq\theta<2\pi}(1-t^2_1)^\beta\int_{t_{2,n}}^{t_1}\frac{|g'(re^{i\theta})|}{(1-r^2)^\alpha}dr\,.\\
\end{split}\end{equation}
Then for any $n\in \mathbb{N}$, we can choose $t_{1,n}$ with $t_{2,n}<t_{1,n}$ such that
\begin{equation}\begin{split}\nonumber
\sup_{0\leq\theta<2\pi}(1-t^2_{1,n})^\beta\int_{t_{2,n}}^{t_{1,n}}\frac{|g'(re^{i\theta})|}{(1-r^2)^\alpha}dr>c-\frac{1}{n^2}\,.
\end{split}\end{equation}
Also, we can choose an angle $\theta_n$ with $0\leq \theta_n<2\pi$ such that
\begin{equation}\begin{split}\nonumber
(1-t^2_{1,n})^\beta\int_{t_{2,n}}^{t_{1,n}}\frac{|g'(re^{i\theta_n})|}{(1-r^2)^\alpha}dr>c-\frac{1}{n}\,.
\end{split}\end{equation}

Define the sequence $\{k_n\}_{n=1}^\infty$ by choosing $k_n$ as the largest positive integer such that $t_{2,n}^{k_n}\geq\frac{1}{3}$ for any $n\in \mathbb{N}$, then since $\lim_{n\rightarrow\infty}\frac{\log\frac{1}{3}}{\log t_{2,n}}=+\infty$, it follows that $\lim_{n\rightarrow\infty}k_n=+\infty$\,.

Now define the sequence $\{s_n(z)\}_{n=1}^\infty$ by
$$s_n(z):=f_n(z)z^{k_n},\quad z\in\mathbb{D}\,,$$
where $f_n(z)$ is defined in the proof of Theorem~\ref{th1}, that is,
$$f_n(z):=\frac{h_n\circ\psi^{-1}_{\eta,\theta_n}(z)}{(1-e^{-2\theta_ni}z^2)^\alpha}\,.$$
Thus,
$$\|s_n\|_{H^{\infty}_{\alpha}}\leq e^{C_2(\gamma,\eta)},$$
and $s_n\rightarrow 0$ uniformly on any compact subset of $\mathbb{D}$\,.

Choosing $n$ large enough such that $t_{2,n}>r_\eta$, then
\begin{equation}\begin{split}\nonumber
\|T_g(s_n)\|_{H^{\infty}_{\beta}}&\geq\left|(1-t_{1,n}^2)^\beta\int_0^{t_{1,n}}f_n(re^{i\theta_n})g'(re^{i\theta_n})r^{k_n}dr\right|\\
&\geq\left|\text{Re}\left((1-t_{1,n}^2)^\beta\int_0^{t_{1,n}}f_n(re^{i\theta_n})g'(re^{i\theta_n})r^{k_n}dr\right)\right|\\
&\geq\left|\text{Re}\left((1-t_{1,n}^2)^\beta\int_{r_\eta}^{t_{1,n}}f_n(re^{i\theta_n})g'(re^{i\theta_n})r^{k_n}dr\right)\right|\\
&\phantom{\geq}-\left|\text{Re}\left((1-t_{1,n}^2)^\beta\int_0^{r_\eta}f_n(re^{i\theta_n})g'(re^{i\theta_n})r^{k_n}dr\right)\right|
\end{split}\end{equation}

First, we have
\begin{equation}\begin{split}\nonumber
\left|\text{Re}\left((1-t^2_{1,n})^\beta\int_0^{r_\eta}f_n(re^{i\theta_n})g'(re^{i\theta_n})r^{k_n}dr\right)\right|&\leq\int_0^{r_\eta}|f_n(re^{i\theta_n})||g'(re^{i\theta_n})|r^{k_n}dr\\
&\leq\|f_n\|_{H^{\infty}_{\alpha}}\int_0^{r_\eta}\frac{|g'(re^{i\theta_n})|}{(1-r^2)^\alpha}r^{k_n}dr\\
&\leq\frac{e^{C_2(\gamma,\eta)}C_3(\eta)}{(1-r^2_\eta)^\alpha}r^{k_n}_{\eta}\,,
\end{split}\end{equation}
where $C_3(\eta):=\sup_{|z|\leq r_\eta}\{|g'(z)|\}$\,.

Next, to estimate another integral, we obtain
\begin{equation}\begin{split}\nonumber
&\phantom{=\;}\left|\text{Re}\left((1-t^2_{1,n})^\beta\int_{r_\eta}^{t_{1,n}}f_n(re^{i\theta_n})g'(re^{i\theta_n})r^{k_n}dr\right)\right|\\
&=\left|(1-t^2_{1,n})^\beta\text{Re}\left(\int_{r_\eta}^{t_{1,n}}\frac{g'(re^{i\theta_n})}{(1-r^2)^\alpha}e^{-i\left(u_n(\psi^{-1}_{\eta,\theta_n}(re^{i\theta_n}))+i\widetilde{u}_n(\psi^{-1}_{\eta,\theta_n}(re^{i\theta_n}))\right)}r^{k_n}dr\right)\right|\\
&=\Bigg|(1-t^2_{1,n})^\beta\int_{r_\eta}^{t_{1,n}}\frac{|g'(re^{i\theta_n})|}{(1-r^2)^\alpha}e^{\widetilde{u}_n\left(\psi^{-1}_{\eta,\theta_n}(re^{i\theta_n})\right)}\times\\
&\phantom{=\;}\text{Re}\left(e^{i\left(\arg(g'(re^{i\theta_n}))-u_n(\psi^{-1}_{\eta,\theta_n}(re^{i\theta_n}))\right)}\right)r^{k_n}dr\Bigg|\\
&\geq\cos(\frac{\pi}{3})e^{-C_2(\gamma,\eta)}\left|(1-t^2_{1,n})^\beta\int_{r_\eta}^{t_{1,n}}\frac{|g'(re^{i\theta_n})|}{(1-r^2)^\alpha}r^{k_n}dr\right|\\
&\geq\cos(\frac{\pi}{3})r_{2,n}^{k_n}e^{-C_2(\gamma,\eta)}\left|(1-t^2_{1,n})^\beta\int_{t_{2,n}}^{t_{1,n}}\frac{|g'(re^{i\theta_n})|}{(1-r^2)^\alpha}dr\right|\\
&\geq\frac{1}{3}\cos(\frac{\pi}{3})e^{-C_2(\gamma,\eta)}(c-\frac{1}{n})\,.
\end{split}\end{equation}
Therefore, it follows that
\begin{equation}\begin{split}\nonumber
\limsup_{n\rightarrow\infty}\|T_g(s_n)\|_{H^{\infty}_{\beta}}&\geq\lim_{n\rightarrow\infty}\left(\frac{1}{3}\cos(\frac{\pi}{3})e^{-C_2(\gamma,\eta)}(c-\frac{1}{n})-\frac{e^{C_2(\gamma,\eta)}C_3(\eta)}{(1-r^2_\eta)^\alpha}r^{k_n}_{\eta}\right)\\
&=\frac{1}{3}\cos(\frac{\pi}{3})e^{-C_2(\gamma,\eta)}c>0\,,
\end{split}\end{equation}
which, according to Lemma~\ref{le3}, is in contradiction with the assumption that $T_g: H^{\infty}_\alpha\rightarrow H^{\infty}_\beta$ is compact.
Thus the proof is complete.
\end{proof}

\begin{remark}\label{re4}
From the proof of Theorem~\ref{th4}, we also see that the condition $$\lim_{t_2\rightarrow1^-}\limsup_{t_1\rightarrow1^-}\sup_{0\leq\theta<2\pi}(1-t^2_1)^\beta\int_{t_2}^{t_1}\frac{|g'(re^{i\theta})|}{(1-r^2)^\alpha}dr=0$$
is sufficient for the compactness of $T_g: H^{\infty}_\alpha\rightarrow H^{\infty}_\beta$ without the requirement that $\log(g')\in \mathcal{B}$.
\end{remark}
\begin{remark}\label{re5}
It should be noticed that Basallote et. al. \cite{BCHMP} had given the following complete characterization of the compactness of $T_g: H^{\infty}_\alpha\rightarrow H^{\infty}_\beta$ for the general symbol $g\in H(\mathbb{D})$ when $\beta>0$: $T_g: H^{\infty}_\alpha\rightarrow H^{\infty}_\beta$ is compact if and only if
$$\lim_{|z|\rightarrow1^-}(1-|z|^2)^{\beta+1-\alpha}|g'(z)|=0\,.$$
\end{remark}

\begin{corollary}\label{cor4}
Let $g\in H(\mathbb{D})$ such that $\log(g')\in \mathcal{B}$. Then the following statements are equivalent:

(1) $T_g: H^{\infty}_0\rightarrow H^{\infty}_0$ is compact;

(2) $T_g: \mathcal{A}\rightarrow \mathcal{A}$ is compact;

(3) $\lim_{t\rightarrow1^-}\sup_{0\leq\theta<2\pi}\int_t^1|g'(re^{i\theta})|dr=0\,.$
\end{corollary}
\begin{proof}
By Theorem~\ref{th4},  $T_g: H^{\infty}_0\rightarrow H^{\infty}_0$ is compact if and only if
$$\lim_{t_2\rightarrow1^-}\limsup_{t_1\rightarrow1^-}\sup_{0\leq\theta<2\pi}\int_{t_2}^{t_1}|g'(re^{i\theta})|dr=0\,.$$
Since $\int_{t_2}^{t_1}|g'(re^{i\theta})|dr$ is increasing as $t_1\rightarrow1^{-}$, it follows that
\begin{equation}\begin{split}\nonumber
\lim_{t_2\rightarrow1^-}\limsup_{t_1\rightarrow1^-}\sup_{0\leq\theta<2\pi}\int_{t_2}^{t_1}|g'(re^{i\theta})|dr&=\lim_{t_2\rightarrow1^-}\sup_{0\leq\theta<2\pi}\limsup_{t_1\rightarrow1^-}\int_{t_2}^{t_1}|g'(re^{i\theta})|dr\\
&=\lim_{t\rightarrow1^-}\sup_{0\leq\theta<2\pi}\int_t^1|g'(re^{i\theta})|dr.
\end{split}\end{equation}

The equivalence of (1) and (2) follows from \cite[Theorem~1.7]{CPPR}.
\end{proof}

Now for the compactness of the companion operators $S_g: H^{\infty}_\alpha\rightarrow H^{\infty}_\beta$, we have the following complete characterization for the general symbol $g\in H(\mathbb{D})$\,.
\begin{theorem}\label{th5}
If $g\in H(\mathbb{D})$ and $0\leq\alpha,\beta$, then the following statements hold:

(1) $S_g: H^{\infty}_\alpha\rightarrow H^{\infty}_0$ is compact if and only if $g=0$\,;

(2) If $0<\beta$, then $S_g: H^{\infty}_\alpha\rightarrow H^{\infty}_\beta$ is compact if and only if
$$\lim_{|z|\rightarrow1^-}(1-|z|^2)^{\beta-\alpha}|g(z)|=0\,.$$
\end{theorem}
\begin{proof}
(1) Anderson et. al. \cite{AJS} had proved that $S_g: H^{\infty}_0\rightarrow H^{\infty}_0$ is compact if and only if $g=0$\,. Since $H^{\infty}_0\subset H^{\infty}_\alpha$ for any $0\leq\alpha$, it follows that $S_g: H^{\infty}_\alpha\rightarrow H^{\infty}_0$ is compact if and only if $g=0$\,.

(2) If $0<\beta$, then $H^{\infty}_\beta=\mathcal{B}^{\infty}_{\beta+1}$\,. Thus the compactness of $S_g: H^{\infty}_\alpha\rightarrow H^{\infty}_\beta$ is equivalent to the compactness of $M_gD: H^{\infty}_\alpha\rightarrow H^{\infty}_{\beta+1}$, where $M_g$ is the multiplication operator and $D$ is the differentiation operator, respectively. Now by choosing $\varphi(z)\equiv z$ in \cite[Theorem~8]{MZ}, the compactness of $M_gD: H^{\infty}_\alpha\rightarrow H^{\infty}_{\beta+1}$ is equivalent to the following condition
$$\lim_{|z|\rightarrow1^-}(1-|z|^2)^{\beta-\alpha}|g(z)|=0\,.$$
\end{proof}

\end{document}